\documentclass[11pt,leqno]{amsproc}
  \usepackage[a4paper,top=3cm,bottom=3cm,left=3cm,right=3cm]{geometry}

\usepackage{amssymb,amsmath}
\usepackage{subfig}
\usepackage[english]{babel}
\usepackage{graphicx}
\usepackage{tabularx}

\usepackage{tikz}
\usetikzlibrary{shapes.geometric,calc}

\newcommand\Z{\mathbb Z}

\def\R{\mathbb R}

\def\Q{\cal Q}
\newcommand{\norm}[1]{\left\lVert#1\right\rVert}
\newcommand{\inner}[2]{\langle #1, #2 \rangle}

\usepackage{mathtools}

\newcommand\C{\mathcal{C}}

\DeclareMathOperator{\Real}{Re}
\let\conjugate\overline

 \newcommand\Span{\operatorname{span}}
 
 \newcommand\dsize{\displaystyle}

\def\C{\mathbb C}

\newtheorem{theorem}{Theorem}
\numberwithin{equation}{section}
\newtheorem{lemma}[theorem]{Lemma}
\renewcommand\r{\rangle}
 \renewcommand\l{\langle}
\newcommand\cal{\mathcal}
 \newtheorem{proposition}[theorem]{Proposition}
\newtheorem{corollary}[theorem]{Corollary}

\begin{document}

\title{Three problems on exponential bases }
\author{Laura De Carli}
\address{Laura De Carli,  Florida International University,
Department of Mathematics,
Miami, FL 33199, USA}
\email{decarlil@fiu.edu}
\author{Alberto Mizrahi}
\address{ Alberto Mizrahi, Florida International University,
Department of Mathematics,
Miami, FL 33199, USA}
\email{amizr007@fiu.edu }
\author{Alexander Tepper}
\address{ Alexander Tepper, Florida International University,
Department of Mathematics,
Miami, FL 33199, USA}
\email{atepp001@fiu.edu}
\begin{abstract}
We consider three special  and significant cases  of the following   problem. Let $D\subset\R^d$ be a  (possibly unbounded) set of finite Lebesgue measure.
Let $E( \Z^d)=\{e^{2\pi i x\cdot n}\}_{n\in\Z^d}$    be the standard exponential basis on the unit   cube  of $\R^d$. 
Find conditions on  $D$  for which  $E(\Z^d)$ is a frame,   a Riesz sequence, or a   Riesz basis for $L^2(D)$.
\end{abstract}
\maketitle
\section{Introduction}
\label{sec:intro}

We are interested in   the following  problem.
Let $D\subset\R^d$ be a   set of Lebesgue measure  $|D|<\infty$.
Let $E( \Z^d)=\{e^{2\pi i n\cdot x}\}_{n\in\Z^d}$    be the standard exponential basis for the unit   cube $  \Q_d =[-\frac 12, \frac 12]^d$.  
Can  $E(\Z^d)$ be frame, or a Riesz sequence, or a   Riesz basis for $L^2(D)$?

We have recalled definitions  and general facts about  frames, Riesz sequences and    Riesz bases   in  Section 2. 
%

Our investigation was motivated  by the following  problems:

\begin{itemize}
\item[{\bf P1.}] {\it  (The broken interval)}. Let    $ J= [0,\alpha)\cup [\alpha+r, L+r)$, with $0<\alpha<\, L$ and $r>0$.  For which values of the parameters   is the set $E(\Z)$ a Riesz basis,  a Riesz sequence or a frame  in $L^2(J)?$
\end{itemize}

 It is easy to verify that  $E(\Z)$ is a frame  on $J$ when $L+r\leq 1$  and it is a Riesz sequence when either $\alpha\ge 1$ of $L-\alpha\ge 1$ (see also Lemma \ref{L-dil-basis}). It is proved in \cite{Laba}  that $E(\Z)$ is an orthonormal  basis for $J$  if and only if     the measure of $J$  is $L=1$ and  the "gap" $  r$ is a non-negative  integer.  

\begin{itemize}
\item[{\bf P2.}] {\it  (The rotated square)}. Let   $Q_h= [-\frac h2, \frac h2]\times [-\frac h2, \frac h2]$ be a  square with side $h>0$. For $\theta\in [0, 2\pi)$, we let $\rho_\theta:\R^2\to\R^2$  be the rotation  $\rho_\theta(x,y)= (x\cos\theta-y\sin\theta, \ x\sin\theta+y\cos\theta)$.
For which values of $\theta $ is   $E(\Z^2)$   a Riesz basis,  a Riesz sequence or a frame on  $ \rho_\theta( Q_h)$?
\end{itemize}

Also the solution to this problem is  trivial only for certain values of the parameters (for example, when $\theta$ is an integer multiple of $\frac \pi 2$).  

\medskip
The next problem  was kindly suggested  by Chun Kit Lai.

\begin{itemize}
\item[{\bf P3.}] {\it  (The translated parallelepiped)}. Let  $P \subset \R^d$ be a parallelepiped with sides parallel to  the vectors $  v_1$, ...,\, $  v_d\in\R^d$.  Find conditions on  these vectors for which  that the set $E(\Z^d)$  is a  Riesz basis,   a  Riesz sequence, or a   frame in $L^2(P)$                                                                                                                                           .
\end{itemize}

We recall that a {\it lattice}   is the image of $\Z^d$ by a linear invertible transformation   \newline $B:\R^d\to\R^d$   
 and we observe that  
Problem 3 is equivalent to the following: {\it for which  lattices $\Lambda=B\Z^d$  is the set  $E(B\Z^d)=\{e^{2\pi i Bn\cdot x}\}_{n\in\Z^d}$  a  Riesz basis, or a  Riesz sequence, or a   frame in  $L^2(\Q_d)$?   }  

Problem 3 is related to certain optimization problems on lattices that have deep applications in computer sciences and in cryptography. See Section  7.1 for details and references.

\medskip
We first prove necessary and sufficient  conditions for which $E(\Z^d)$ is a Riesz sequence   or a frame on a given domain $D\subset \R^d$ and then we completely solve   Problems 1, 2 and 3.
 
 We      let \begin{equation}\label{e-Phi} \Phi(x)=\sum_{m\in\Z^d}   \chi_D(x+m), \end{equation} where      $\chi_D$  denotes   the characteristic function of   $D$. Note that $\Phi(x)$ only takes  non-negative integer values.  Our first result is  the following  
  
\begin{theorem}
 \label{T-Riesz }
$E(\Z^d)$ is a Riesz sequence in $L^2(D)$
if and only if  there exist  constants $0<A\leq B <\infty$ for which $A\leq  \Phi(x) \leq B$ for a.e.   $x\in \Q_d $.   
 
 \end{theorem}
That is, we prove that $E(\Z^d)$ is a Riesz sequence in $L^2(D)$ if and only if   the  integer translates of $D$  (i.e., the sets   
    $D+n=\{x+n,\, x\in D\}$, with $n\in\Z^d$) cover  $\R^d $  with the possible exception of  a set of measure zero.  
     
     \medskip                            
It is interesting to compare  Theorem \ref{T-Riesz }  with results in \cite{AAC, K, GL}. 
In these papers the authors consider  domains  that {\it multi-tile } $\R^d$, i.e.  bounded measurable sets  $S \subset\R^d$   for which   there exist a set  of translations $\Lambda $  and an  integer $h>0$    such that  $\sum_{\lambda\in\Lambda} \chi_{S+\lambda}(x)\equiv h$   a.e.; if $h=1$,  we  say that $S$  {\it  tiles} $\R^d$.
It is proved in   \cite[Theorem 1]{GL} and in  \cite[Theorem 1]{K} that bounded domains  that   multi-tile  $\R^d$ with a lattice of translation  have an exponential basis; in the recent \cite[Theorem 4.4] {AAC} the converse of \cite[Theorem 1]{K} is proved.  

If $\Phi$ is   as in \eqref{e-Phi}  and     $\Phi(x)\equiv k$ a.e., then   $D$    multi-tiles $\R^d$ with lattice of translations $\Z^d$. By Theorem \ref{T-Riesz },   $E(\Z^d)$ is a Riesz sequence on $L^2(D)$; when $D$ is bounded,   it is shown in  \cite[Theorem 1]{K}   that   $E(\Z^d)$ can be completed to an exponential  basis for $L^2(D)$, but when $D$ is not bounded an example in  \cite{AAC} shows that that may not be possible.

 \medskip

Next,  we investigate conditions for which  $E(\Z^d)$ is  a frame on   $D $.  
The following   result is  proved in  \cite[Lemma 2.10]{GLi}. See also    the recent \cite[Theorem 2]{BHM}.
   \begin{theorem} \label{T-frame}
    $E(\Z^d)$ is a frame on $L^2( D)$ if and only if   for every   $ m,   s\in\Z^d$, with  $  m\ne   s$, 
       we have that  \begin{equation}\label{e-ass-fr}
       |(D+  m)\cap (D+ s)|=0.
       \end{equation}
\end{theorem}
In other words,   
  $E(\Z^d)$ is a frame in $L^2(D)$ if and only if the  integer translates of $D$ only overlap  on sets of measure zero. Equivalently, $E(\Z^d)$ is a frame on $L^2( D)$ if and only if $ \Phi(x) \leq 1$ for a.e. $x\in\R^d$.  

\medskip
We finally prove the following 

\begin{theorem}\label{T-basis} Assume that $|D|=1$. The following are equivalent
in $L^2(D)$:
\begin{itemize}\item [a)]
$E(\Z^d)$ is a frame 
\item [b)] $E(\Z^d)$ is a complete 
\item [c)] The integer translates of  $D$  tile   $\R^d$.
\item[d)] $E(\Z^d)$ is an orthonormal Riesz basis    
\item[e)] $E(\Z^d)$ is a Riesz sequence 

\end{itemize}
\end{theorem}

We recall that a set   $\{w_i\}_{i\in I}$ is {\it complete} in a Hilbert space $ (H ,\  \l \ , \ \r_H)$ if and only if  $\l u, w_i\r_H=0$ for every $i\in I$  implies $u=0$.  A frame is complete but   the converse is not necessarily true.

B. Fuglede proved in \cite{F} that  if $\Lambda=A\Z^d$ is a lattice in $\R^d$,  the set $E(A\Z^d)$  is  an orthogonal exponential basis in $L^2(D)$  if and only if   $\{D +\mu\}_{\mu\in( A^t)^{-1} \Z^d }$ tiles $\R^d$.  Here, $(A^t)^{-1}$  denotes the inverse of the transpose of  $A$.
Thus,  the equivalence of  c) and d) in   Theorem \ref{T-basis} is  a special case of Fuglede's  theorem. 
The connections between tiling  and   exponential bases are deep and interesting and  have been intensely investigated. We refer the reader to  the introduction and to the references cited in    \cite{BHM}. See also \cite{K2}.

\medskip
 This paper is organized as follows: 
 in Section 2 we present  preliminary definitions and known results. We prove Theorems    \ref{T-Riesz }, \ref{T-frame} and \ref{T-basis} in Sections 3 and 4. We solve Problems 1, 2 and 3 in Sections  5, 6 and 7.
 
 \medskip \noindent
 {\it Acknowledgments.} The first author wishes to thank C.K. Lai for  bringing Problem 3 to our attention and E. Hernandez for  stimulating discussions that helped us improve the quality of this paper. We also  wish thank the anonymous referees for their thorough reading of our manuscript and for their valuable  suggestions.

\section{ Preliminary and notation}

  We denote with $x\cdot y= x_1y_1+\,...\,+ x_dy_d$  the   inner product of    $x=(x_1, ...,\, x_d)$, $y=(y_1, ...,y_d)\in\R^d$.

  We let  $\|f\|_2=\left(\int_{\R^d} |f(x)|^2 dx\right)^{\frac 12}$   be the  standard norm in  $L^2(\R^d)$; we let ${\bf c}=\{c_j\}_{j\in\Z^d}$ and we denote with   $\|{\bf c}\|_{\ell^2}=(\sum_{j\in\Z^d } |c_j|^2)^{\frac 12}$  the  standard norm in $\ell^2(\Z^d)$.   We denote with $\l f ,\, g \r_{2}=\int_{\R^d} f(x)\bar g(x)dx$   the  inner product   in $L^2(\R^d)$. When there is no ambiguity we will use the same notation also for the inner product in $\ell^2(\Z^d)$.

The Fourier transform of  a function $f\in L^2(\R^d)\cap L^1(\R^d)$ is $\hat f(x)=\int_{\R^d} f(t) e^{-2\pi i x\cdot t} dt. $

 We will  often say that a family of sets $\{D_\lambda\}_{\lambda\in\Lambda}$ {  covers}   $\R^d$  with the understanding that  $\R^d-\cup_{\lambda\in\Lambda} D_\lambda $ may be a  nonempty set of measure zero.  
 
We use the notation   $\tau_w$ to denote the translation operator $g \to  g  (\cdot+w)$.
 
\subsection{Frames and Riesz bases}
 
 We have used the excellent textbooks \cite{Heil} and \cite{Cr} for most of the definitions  and preliminary results presented in this section. 
 
 Let $H$ be a  separable Hilbert space  with inner product $\langle\ ,\ \rangle $  and norm $||\ ||=\sqrt{\l \ , \    \r} $. 
A sequence of vectors ${\mathcal V}= \{v_j\}_{j\in\Z} \subset H $   is a
 {\it frame} if
there exist  constants $0< A, \ B<\infty$    such that the following inequality holds  for every $w\in H$.
\begin{equation}\label{e2-frame}
 A||w||^2\leq  \sum_{j\in\Z} |\l  w, v_j\r |^2\leq B ||w||^2.
\end{equation}

  We say that  ${\cal V}$  is a   {\it  tight frame } if $A=B$   and   is a {\it Parseval frame} if $A=B=1$. 


The left inequality in \eqref{e2-frame} implies that  ${\cal V}$ is     complete    in $H$  but it may not be linearly independent. A {\it Riesz basis} is a linearly independent frame.

An equivalent definition of  Riesz basis is the following: the set 
${\mathcal V}$ is a  {\it Riesz sequence}  if   there exists constants $0<A\leq B <\infty$ such that, for every finite set of coefficients $ \{a_j\}_{j\in J}\subset\C $,  
we have that
\begin{equation}\label{e2- Riesz-sequence}
 A  \sum_{j\in J}   |a_j|^2   \leq  \left\Vert \sum_{j\in J}  a_j  v_j \right\Vert^2  \leq B \sum_{j\in J} |a_j|^2, 
\end{equation}
and it is  {\it Riesz basis} if it also satisfies  \eqref{e2-frame}.  
If ${\mathcal V}$ is a Riesz basis, the  constants $A$ and $B$ in \eqref{e2-frame} and \eqref{e2- Riesz-sequence}  are the same (see \cite[Proposition  3.5.5]{Cr}).

 An orthonormal  basis   is a Riesz basis; we can write   $w=\sum_{j\in\Z} \l v_j,  \, w\r v_j$ for every $v\in H$ and   this  representation formula  yields  the following important identities: for every $    w,\ z\in  H$, 
\begin{equation}\label{e-Planch}
||w||^2= \sum_{n\in\Z } |\l v_n,\, w\r|^2, \quad \l w, z\r=   \sum_{n\in\Z }  \l v_n,\, w\r\overline{\l v_n, z\r}. 
\end{equation}
The following useful  proposition can be found in  \cite[Prop. 3.2.8]{Cr}.
 \begin{proposition}\label{prop-C}  A sequence of unit vectors in $H$  is a Parseval frame if
and only if it is an orthonormal Riesz basis.
\end{proposition}
%
%

\medskip
Let $D\subset \R^d$ be a measurable set, with $|D|<\infty$.
An {\it exponential basis} of $L^2(D)$ is a Riesz basis made of  functions in the form of $  e^{2\pi i   x\cdot \lambda  }$, where    $ \lambda \in\R^d$.
Exponential bases are important in the applications because they allow one to represent   functions in $L^2(D)$ in a stable manner, with coefficients that are easy to calculate.  
%

The following lemma is easy to prove   (see e.g   \cite[Prop. 2.1]{DK}).
\begin{lemma}\label{L-dil-basis} Let $D_1\subset D\subset D_2$ be measurable sets of $\R^d$ , with $|D_2|<\infty$. 
Let  ${\cal V}=\{ e^{2\pi i   x\cdot \lambda_n  }\}_{n\in\Z}$  be  Riesz basis of $L^2(D)$  with frame constants $0<A\leq B<\infty$; then, ${\cal V}$ is a Riesz sequence on $L^2(D_2)$ and a frame on  $L^2(D_1)$
 with the same frame constants. 
\end{lemma}

 \subsection{The Beurling density} 
In \cite{B, B1} A. Beurling   characterized sampling sets by means of their density.
 
For $h > 0$ and $x  \in\R^d$,  we let $Q _h(x)$ denote the  closed cube centered at $x$ with
side length $h$.
Let $\Lambda=\{\lambda_j\}_{j\in\Z}\subset  \R^d$ be    {\it uniformly discrete}, i.e., we assume that
$|\lambda_j-\lambda_k|\ge \delta>0$
  whenever $\lambda_j\ne\lambda_k$.  Following \cite{CDH} we denote with
\begin{align*} {\mathcal D}^+(\Lambda) &= \limsup_{h\to \infty} \frac{\sup_{x\in\R^d} |\Lambda\cap Q_h(x)|}{h^d}\\ 
{\mathcal D}^-(\Lambda) &= \liminf_{h\to \infty} \frac{\inf_{x\in\R^d} |\Lambda\cap Q_h(x)|}{h^d}
\end{align*}
 the upper and lower  density of $\Lambda$.  
  If ${\mathcal D}^-(\Lambda)={\mathcal D}^+(\Lambda)$  we say that $\Lambda$ has {\it uniform Beurling density} ${\mathcal D}(\Lambda)$.
 
Theorem \ref{T-density-exp} below is a generalization of  theorems of Landau  and Beurling \cite{B, Landau} in dimension $d\ge 1$. See also \cite{NO} and  \cite[Sect. 2]{S}. 
 
\begin{theorem}\label{T-density-exp}
 
If $E(\Lambda)=\{e^{2\pi i \lambda_j\cdot x}\}_{j\in\Z}$ is a frame in $L^2(D)$, then  ${\mathcal D}^-(\Lambda)\ge |D|$. 
If  $E(\Lambda)$ is a Riesz sequence in $L^2(D)$, then ${\mathcal D}^+(\Lambda)\leq |D|$.
\end{theorem}

Thus, a necessary condition for   $E(\Lambda)$ to be  a Riesz basis in $L^2(D)$  is  that ${\mathcal D} (\Lambda)= |D|$.  In the special case  of $\Lambda=\Z^d$ we have the following
\begin{corollary}\label{C-density}
If $E(\Z^d)$ is a frame in $L^2(D)$ then $|D|\leq 1$; if  $E(\Z^d)$ is a Riesz sequence   in $L^2(D)$ then $|D|\ge 1$.
\end{corollary}

\subsection{Shift invariant spaces}
We let 
$$
V^2(\varphi):=\overline{\Span \{\tau_k\varphi \}_{k\in\Z^d  } } 
$$
where  $\varphi\in L^{2}(\R^d)$   and ``bar'' denotes the closure in $L^2(\R^d)$.
The space $  V^2(\varphi)$ is {\it shift-invariant}, i.e.  if $f\in V^2(\varphi)$  then also $\tau_m f\in V^2(\varphi)$ for every $m\in\Z^d$. 
Shift-invariant spaces of functions appear naturally in signal theory and in other branches of applied sciences.
Following \cite{Aldroubi}, \cite{Christiansen}, \cite{CCS} 
 we say that  the translates $\{ \tau_k\varphi \}_{k\in\Z^d}$ form a Riesz basis in $V^2 ( \varphi ) $ if there exist constants $0<A,\ B<\infty$ such that, for every   finite set of coefficients ${\bf d}= \{d_j\}  \subset\C $, we have that 
\begin{equation}\label{E-p-basis-2}
A \|{\bf d}\|_{\ell^2} \leq \| \sum_j d_j  \tau_j\varphi  \|_{2} \leq B\|{\bf d}\|_{\ell^2}.
\end{equation}
If \eqref{E-p-basis-2}  holds,  then  $
V^2(\varphi )= \left\{ f= \sum_{k\in\Z^d}d_k  \tau_k\varphi, \ {\bf d} \in \ell^2\  \right\},
$
and the sequence $\{d_k  \}_{k\in\Z^d}$ is uniquely determined by $f$.

\medskip
The following  theorem is well known: see e.g \cite{JM} or \cite[Prop. 1.1]{AS}.  
\begin{theorem} \label{T-basis2}
The set $\{ \tau_m\varphi \} _{m\in\Z^d}$ is a  Riesz basis in $V^2(\varphi)$ with frame constants $0<A,\ B<\infty$   if and only if, 
\begin{equation}\label{e1}
A= \inf_{y\in \Q_d}\sum_{m \in\Z^d} |\hat \varphi(y+m)|^2 \leq \sup_{y\in \Q_d}\sum_{m \in\Z^d} |\hat \varphi(y+m)|^2 = B.
\end{equation}
\end{theorem}

\section{ Proof of Theorem \ref{T-Riesz }}

 Let $\ell^2_0(\Z^d)\subset  \ell^2(\Z^d)$   be the  set of    sequences ${\bf a}=(a_n)_{n\in\Z^d}$ such that $a_n=0  $ whenever $|n|\ge N$, with $N=N({\bf a})\ge 0$.
 Let
$ S({\bf a})  =    \sum_{n\in\Z^d} a_n e^{2\pi i n\cdot x}$. 
Recall that $E(\Z^d)$ is a Riesz sequence in $L^2(D)$  if and only if  there exists  constants $0<A,\ B<\infty$ such that
\begin{equation}\label{ineq-S}
A\|{\bf a}\|_2^2\leq \|S({\bf a}) \|_{L^2(D)} ^2\leq B\|{\bf a}\|_2 ^2
\end{equation}
for every  ${\bf a}\in \ell^2_0(\Z^d).$ 
We gather: 
\begin{align}\nonumber
\|S({\bf a}) \|_{L^2(D)} ^2= &  \int_D \left|\sum_{n\in\Z^d} a_n e^{2\pi i n\cdot x}\right|^2 dx= \int_D\left(\sum_{n, m\in\Z^d} a_n \overline{a_m}\, e^{2\pi i (n-m)\cdot x}\right)dx
\\\label{Sa}
=&  \sum_{n, m\in\Z^d} a_n \overline{a_m}\, \int_D e^{2\pi i (n-m)\cdot x} dx
=  \sum_{n, m\in\Z^d} a_n \overline{a_m} \widehat{\chi_D}(n-m).
\end{align}
 Let $T_D$ be the operator,  initially defined in $\ell^2_0(\Z^d)$, as: 
\begin{equation}\label{e-TD} T_D({\bf a})_m= \sum_{n \in\Z^d} a_n   \widehat{\chi_D}(n-m), \ m\in\Z^d.
\end{equation}
The calculation above shows that 
$
\|S({\bf a}) \|_{L^2(D)} ^2=\l T_D({\bf a}), \ {\bf a}\r_2,  
$
where  $\l\, , \,\r_2 $  denotes the inner product  in $\ell^2(\Z^d)$. We can easily verify that    $T_D({\bf a})$ is self-adjoint    and, in  view of \eqref{Sa},    that 
  $ \l T_D({\bf a}),\ {\bf a}\r_2\ge 0$   for every ${\bf a}\in\ell^2_0(\Z^d)$; thus, \eqref{ineq-S} holds if and only if 
\begin{equation}\label{e-1}
A\|{\bf a}\|_2\leq \l T_D({\bf a}), \ {\bf a}\r_2 \leq B\|{\bf a}\|_2,\quad {\bf a}\in \ell^2_0(\Z^d).
\end{equation}
To prove \eqref{e-1} we need the following 
\begin{lemma}\label{L-Haase}
Assume that $\dsize ||T_D||_{\ell^2\to\ell^2}= \sup_{||{\bf a}||_2=1}|| T_D ({\bf a})||_2 <\infty$.
The   inequality below holds for every ${\bf a}\in \ell^2_0(\Z^d)$ such that  $||{\bf a} ||_2=1$.
\begin{equation}\label{e2}
  \displaystyle \frac{ ||T_D({\bf a})||_2^2}{  ||T_D||_{\ell^2\to\ell^2}}\leq \l T_D({\bf a}),\, {\bf a}\r_2 \leq  ||T_D||_{\ell^2\to\ell^2}
.  
\end{equation}
\end{lemma}

\begin{proof}[Proof of Theorem \ref{T-Riesz }]
Let $\Phi(x) $ be as in \eqref{e-Phi}. We show that  if there exist constants   $0<A' \leq   B'  <\infty$ such that  $A' \leq \Phi(x)\leq B' $ a.e. in $\Q_d$, then $E(\Z^d)$ is a Riesz sequence in $L^2(D)$.
Since $\Phi(x)=\sum_{m \in\Z^d}   \chi_D(x+m) =\sum_{m \in\Z^d}   |\chi_D(x+m)|^2 $, 
by Theorem \ref{T-basis2}    the set  $\{ \tau_m\hat \chi_D \} _{m\in\Z^d}$ is a  Riesz basis of $V^2(\hat \chi_D )$ with frame constants $0<A' ,\ B' <\infty$.
In  view of \eqref{E-p-basis-2} and \eqref{e-TD}, the inequality 
\begin{equation}\label{e-prime} 
A' \|{\bf a}\|_2\leq \|T_D{\bf a}\|_2^2\leq B' \|{\bf a}\|_2
\end{equation}  
holds  for every ${\bf a}\in \ell^2_0(\Z^d)$. By Lemma \ref{L-Haase}, we have \eqref{e-1}, as required.

 If $E(Z^d)$ is a Riesz sequence on $D$, we argue as in the proof of  \cite[Theorem 3.1]{selvan}.   Using Plancherel's identity and the Poisson summation formula, from \eqref{Sa} we obtain
\begin{align}\nonumber
||S({\bf a}) ||_{L^2(D)} ^2 &= \sum_m {\Big \vert} \sum_{n\in\Z^d} a_n\widehat{\chi_D}(n-m) {\Big \vert} ^2    =   \int_{\Q_d}   {\Big\vert} \sum_{m\in\Z^d}   \sum_{n\in\Z^d} a_n\widehat{\chi_D}(n-m) e^{2\pi i x\cdot m} {\Big\vert}^2dx
\\ \nonumber & =\int_{\Q_d}  {\Big\vert} {\Big( } \sum_{n\in\Z^d} a_n e^{2\pi i x\cdot n} {\Big )}\sum_{m\in\Z^d}      \widehat{\chi_D}(n-m) e^{2\pi i x\cdot (m-n)} {\Big\vert}^2dx
\\ \label{new e}& =
\int_{\Q_d}  \, {\Big\vert}\sum_{n\in\Z^d} a_n e^{2\pi i n\cdot x}{\Big\vert}^2\, |\Phi(x)|^2  dx.
\end{align}
By assumption,  the integral in \eqref{new e}  is finite and so   so  $\Phi(x) <\infty$ a.e..
To show that  $\Phi(x) >0$ a.e. we argue by 
 contradiction: suppose that there exists   $\Omega\subset D$, with $|\Omega|>0$,  where $\Phi(x)\equiv 0$. We can assume that $\Omega\subset \Q_d$.    Since $E(\Z^d)$ is a Riesz basis in $L^2(\Q_d)$, we can write    $\chi_{\Q_d-\Omega}(x)=\sum_{n\in\Z^d} b_ne^{2\pi i n\cdot x}$, with $\vec b\in \ell^2(\Z^d)$.  Thus, $\int_{\Q_d} |\Phi(x)|^2 \, \left|\sum_{n\in\Z^d} b_n e^{2\pi i n\cdot x}\right|^2dx =0$  which, together with \eqref{new e}, contradicts \eqref{ineq-S}.
 \end{proof}

 
\begin{proof}[Proof of Lemma \ref{L-Haase}]
The right   inequality in  \eqref{e2}  is \cite[Theorem 13.8]{Haase} so we only need to  prove the left inequality. Let  $\alpha= \sup_{||{\bf a}||_2=1}|\l T_D({\bf a}),\, {\bf a}\r|$ and   $U= \alpha I-T_D$, where $I$ is the identity operator in $\ell^2(\Z^d)$. It is easy to verify that $U$ is positive and that   
  \begin{equation}\label{e-id3}
   T_D\,U\,T_D+\,U\,T_D\,U\,=\alpha^2T_D-\alpha T_D^2.
  \end{equation}
  The operators  $T_D\,U\,T_D$ and $\,U\,T_D\,U\,$ are   positive too; indeed, for every ${\bf a}\in \ell^2$,  we have that  $\l T_D\,U\,T_D {\bf a}, \, {\bf a}\r_2= \l \,U\,(T_D{\bf a}), \ T_D{\bf a}\r_2\ge 0$  and $\l \,U\,T_D\,U\,{\bf a} , \, {\bf a}\r_2= \l T_D(\,U\,{\bf a}), \ \,U\, {\bf a}\r_2 \ge 0$ because $T_D$ and $\,U\,$ are both positive. By \eqref{e-id3}, also the operator $\alpha T_D-  T_D^2$ is positive. For every  ${\bf a}\in \ell^2$ with $||{\bf a}||_2=1$, we have that 
  $$
  \l (\alpha T_D-  T_D^2) {\bf a}, \  {\bf a}\r_2= \alpha  \l T_D{\bf a}, \ {\bf a}\r_2- \l T_D^2{\bf a}, \ {\bf a}\r_2
  = \alpha  \l T_D{\bf a}, \ {\bf a}\r_2-||T_D{\bf a}||_2^2\ge 0 
  $$
and the left inequality in \eqref{e2} is proved.
  \end{proof}

  \medskip
 \noindent
 {\it Remark.}  From the  identity \eqref{new e}  it follows  that the constants  $A$ and $B$ in \eqref{ineq-S} are the minimum and maximum of   $\Phi(x)$  on the unit square $\Q_d$. 
%
 Thus, 
 $A $ and $B $ are   integers.
 
When  $|D|=1$,    Theorem \ref{T-basis} shows that  $E(\Z^d)$ is a Riesz sequence if and only  the integer translates of $D$ tile $\R^d$,   and so $A =B =1$. In general, if   $k \leq |D|< k+1$ for some positive integer $k$, we can easily verify that  the integer translates of  $D $   cover   $\R^d$ $k$ times but not $k+1$ times.    
Thus,   $A  \leq |D|$ and $B  \ge |D|$. 
 
 \section{Proof of Theorem \ref{T-frame}  }

\medskip
Let $D\subset \R^d$ be measurable, with $|D|\leq 1$.  By Lemma \ref{L-dil-basis}, 
 the theorem is trivial when $D\subset \Q_d $, so  we assume that $ D- \Q_d    $  has positive measure. Let  $D_1$, ...,\, $D_N ,\,...$ be a (possibly infinite) family of disjoint sets of positive measure such that      $D-\Q_d= \cup_{j} D_j$.  We can choose the $D_j$ in such way that, for certain vectors  $v_1$, ... $v_N,\, ... \in \Z^d$, we have that $D_j+v_j\subset  \Q_d.  $  Let   $D_0= D\cap \Q_d $ and $v_0=0$.
 We prove the following 
\begin{lemma} \label{L-frame2}
 
    $E(\Z^d)$ is   frame for $L^2( D)$ if and only, for every  $  v  \in\Z^d$ and every $k\ne j$,   
       \begin{equation}\label{e-assumptions-frame}|(  v +D_j)\cap  D_k |= 0.  \end{equation}

\end{lemma}
 
 It is easy to verify that   \eqref{e-assumptions-frame} is equivalent to  \eqref{e-ass-fr},  and so Theorem \ref{T-frame} is equivalent to  Lemma \ref{L-frame2}.

\begin{proof} 

Assume  that  $|(D_1 + v) \cap D_0|>0 $ for some $v\in\Z^d$  (the proof is similar in the other cases). We can assume without loss of generality that   $D_1 + v \subset \Q_d $  (see Figure 1); otherwise we let  $D_1= D_1' \cup D_1''$, with    $D'_1+v\subset \Q_d$ and we replace $D_1$ with $D_1'$.  We show that $E(\Z^d)$ is not a frame on  $L^2(D)$.

\begin{figure}[h!]
  \begin{center}
     \begin{tikzpicture}

 \fill[gray!20] (-1.5,0)--  (-1.5, 3)-- ( -.5,3)--(0, 4)--(.5,3)--(1.5,3)--(1.5, 0)--( .7, 0)-- (.7, .7)--(-.7, .7)-- (-.7, 0)--(-1.5,0);
 \fill[gray!20]  ( -.5,3)--(0, 4)--(.5,3)-- (-.5,3);
 \draw[black]( -.5,0)--(0, 1)--(.5,0)-- (-.5,0);
 \fill[gray!60]  ( -.5,0)--(0, 1)--(.5,0)-- (-.5,0);
\draw[  black] (-1.5,0)--  (-1.5, 3)-- ( -.5,3)--(0, 4)--(.5,3)--(1.5,3)--(1.5, 0)--( .7, 0)-- (.7, .7)--(-.7, .7)-- (-.7, 0)--(-1.5,0);;
 \draw[thick, black,  dashed] (-1.5,0)--(-1.5,3)--(1.5, 3)--(1.5,0)--(-1.5, 0);

 \draw [-> ]  (0,3)-- (0,1.5);
 \draw (.2 ,2.2) node [black ] {$   v$};
 \draw (-.5 ,1.5) node [black ]   {$D_0$};
  \draw (0 ,3.3) node [black ]   { \small{$D_1$}};
   \draw (0 , .2) node [black ]   { \tiny{$  D_1\!+\!v$}};
 \draw (2 , 1.5) node [black ]   {$\Q_d $};

  \end{tikzpicture}
 \caption{ } 
   \end{center}
 \end{figure}

  Every $f\in L^2(D)$ can be written as $f = f_0 + f_1$ where $f_0 = f \chi_{D-D_1}$ and $f_1 = f \chi_{D_1}$.  
Recall that     $\tau_w g(x)= g  (x+w)$.
    It follows  that
     \begin{align}\nonumber
        |\inner{e^{2\pi i n\cdot x}}{f}_{L^2(D)}|^2
        &= |\inner{e^{2\pi i n\cdot x}}{f_0}_{L^2(D-D_1)} + \inner{e^{2\pi i n\cdot x}}{f_1}_{L^2(D_1)}|^2 \\\label{1}
        &= |\inner{e^{2\pi i n\cdot x }}{f_0}_{L^2(D-D_1)} + \inner{e^{2\pi i n\cdot (x-v)}}{\tau_{-  v}f_1 }_{L^2( D_1+v))}|^2 \\\nonumber
        &= |\inner{e^{2\pi i n\cdot x}}{f_0}_{L^2(\Q_d )} + \inner{e^{2\pi i n\cdot x}}{\tau_{-v}{f_1}}_{L^2(\Q_d )}|^2 \\\nonumber
        &= |\inner{e^{2\pi i n\cdot x}}{f_0}_{L^2(\Q_d )}|^2 + |\inner{e^{2\pi i n\cdot x}}{\tau_{-v}{f_1}}_{\Q_d )}|^2 \\\nonumber & \quad + 2\Real{\left(\inner{e^{2\pi i n\cdot x}}{f_0}_{L^2(\Q_d )} \conjugate{\inner{e^{2\pi i n\cdot x}}{\tau_{-v}{f_1}}_{L^2(\Q_d )}}\right)}.
    \end{align}
We have used the change of variables $x\to x-v$ in the second inner product in \eqref{1} and  the fact that  $e^{2\pi i n\cdot v}=1$. Since $E(\Z^d)$ is an orthonormal basis in $\Q_d $,  the identities  \eqref{e-Planch} in Section 2  and the calculation above  yield
    \begin{equation}\label{e-id1}
    \begin{split} 
        \sum_{n  \in \Z^d} |\inner{e^{2\pi i n\cdot x}}{f}_{L^2(D)}|^2
        &=  \norm{f_0}_{L^2(\Q_d )}^2 + \norm{ \tau_{-v} f_1 }_{L^2(\Q_d )}^2 \\ & + 2\Real{ \inner{f_0}{\tau_{-v}{f_1}}_{L^2(\Q_d )} }.
    \end{split}
    \end{equation}
    If we let $B= (D-D_1)\cap (D_1 +v)$   
  we can choose $ f=f_1+f_0$, with $ f_1(x)= \chi_{B}(x+v)$   and $f_0(x) =- \chi_B(x) $; from \eqref{e-id1} it readily follows that  
    $\sum_{n  \in \Z^d} |\inner{e^{2\pi i n\cdot x}}{f}_{L^2(D)}|^2  =0$, which contradicts 
  \eqref{e2-frame}.
 
  \medskip
   
We now assume that  $|(w+D_j)\cap D_k|=0$ for every $k\ne j$ and every $w\in\Z^d$; we prove that  $E(\Z^d)$ is a tight frame in $L^2(D)$. 
We  assume for simplicity that    $ D_1+v  \subset \Q_d $ for some   $v \in\Z^d$. Let $f=f_0+f_1$ be as in the first part of the proof. By assumption,  $|(D_1+v)  \cap (D-D_1)|  =0$,   and so \eqref{e-id1} yields
  \begin{align*}\sum_{n  \in \Z^d} |\inner{e^{2\pi i n\cdot x}}{f }_{L^2(D)}|^2& = \norm{f_0}_{L^2(\Q_d )}^2 + \norm{ \tau_{-v} f_1 }_{L^2(\Q_d )}^2 
   \\ &=\norm{f  \chi_{D-D_1}}_{L^2(\Q_d )}^2 + \norm{   f \chi_{D_1 +v}}_{L^2(\Q_d )}^2
    = ||f||_{L^2(D)}.
   \end{align*}
 Thus, $E(\Z^d)$ is a tight frame in $L^2(D)$ as required.
\end{proof}
 
 \medskip
   The  proof of Theorem \ref{T-frame} shows that if  \eqref{e-assumptions-frame} is not  satisfied,      we can produce a function $f\in L^2(D)$ for which $\l f,\, e^{2\pi i x\cdot n}\r_{L^2(D)}=0$ for every $n\in\Z^d$, and so $E(\Z^d)$ is not complete.   This observation proves  the following: 
 
 \begin{corollary}\label{C-complete}
 $E(\Z^d)$ is complete in $L^2(D)$ if and only if  the integer translates of $D$ intersect on sets of measure $0$.
 \end{corollary}

  \begin{proof}[Proof of Theorem \ref{T-basis}] 
  We have proved that  a) $\iff$ b);  we show that $b)\iff c)$.
  By Corollary \ref{C-complete}, 
  $E(\Z^d)$ is  complete in $L^2(D)$ if and only if  the integer translates of $D$ overlap only on sets of measure zero.  Thus, c) $\Rightarrow$ b). Let us prove that  b) $\Rightarrow$ c);  let   $D_0,\, D_1$, ...,\, $D_N ,\,...$  and $v_0, \,v_1, \, ...,\, \, v_N,\, ...$ be   as  in the proof of Lemma \ref{L-frame2}. 
  Since  $|(D_j+v_j)\cap (D_k+v_k)|=0 $ when $k\ne j$, and  
  \begin{equation}\label{e-Q-D}
  1= |\Q_d|=|D|= |D_0| +  \sum_j |D_j+v_j|  \end{equation}
   necessarily  $ \cup_j  (D_j+v_j) =\Q_d$   and   the integer translates of $D$ tiles $\R^d$.   
   
   By Fuglede's theorem, c) $\iff$ d). Clearly d) $\Rightarrow$  e);  to finish the proof of the theorem we  show that e) $\Rightarrow$ c).  By Theorem \ref{T-Riesz }, the integer translates of $D$ cover $\R^d$; thus, $ \cup_j  (D_j+v_j) =\Q_d$ and from \eqref{e-Q-D} follows that    the $D_j+v_j$'s can only intersect   on set of measure zero.  Thus, the integer translated of $D$ can only intersect on sets of measure zero and   c)   is proved.

  \end{proof} 
  \section{The broken interval}
  
In this section we solve  the first problem stated in the introduction. We  let $J= [0,\alpha)\cup [\alpha+r, L+r)\subset \R$, with   $0<\alpha<L$ and   $r>0$. 
  
  By Lemma \ref{L-frame2} and Theorem \ref{T-Riesz }, $E(\Z)$ is a frame on $L^2(J)$ if and only if  the integer translates of $J$ do not overlap in $[0,1]$ and it is a Riesz sequence if and only if the integer translates of $J$ cover $\R$.   

\medskip

Let $[r]$ be the integer part of $r$, i.e., the largest integer $n \leq r$; let   $\{r\}=   r - [r]$   be the fractional part of $r$.
We prove the following
\begin{theorem}\label{T-riesz-interval}

a)   $E(\Z)$ is a frame on $J$ if and only if    $L + \{r\} \leq 1$.
 \\
b) $E(\Z)$ is a Riesz sequence on   $J $ if and only if one of the following is true: \begin{itemize}
\item[i)]
	$\alpha \geq 1$ or $L - \alpha \geq 1$ 
\item[ ii)]
	$\{r\} = 0$ and $L \geq 1$
\item[ iii)] $1 \leq L < 2$ and $L + \{r\} \geq 2$;
\end{itemize}
\medskip

\end{theorem}

 To prove  b) we will need the following 
 \begin{lemma}\label{L-frac-part}
The  integer translates of $J$ cover  $\R$ if and only if the integer translates of $J' = [0,\alpha) \cup [\alpha + \{r\}, L + \{r\})$ cover  $\R$.  
 \end{lemma}
 
 \begin{proof} 
 If the integer translates of $J$ cover   $\R$ then for   $ x\in\R$, there is an integer $m$ such that either $x \in (m, \alpha +m)$ or $x \in (\alpha + r + m,\ L + r + m)$. If $x \in (m, \alpha + m)$ then clearly $x\in J'+m$ as well. If $x \in (\alpha + r + m,\ L + r + m)$, then $x \in (\alpha + \{r\} + [r] + m,\ L + \{r\} + [r] + m)$, i.e. $x$ is in the translation of $J'$ by $[r]+m$. The converse is similar.   
  \end{proof}

\begin{proof}[Proof of Theorem \ref{T-riesz-interval}]

By Theorem \ref{T-frame} and Lemma \ref{L-frac-part},  $E(\Z)$ is a frame on $J$ if and only if    the integer translates of $[0,\alpha)\cup[\alpha + \{r\}, L + \{r\})$ do not intersect in $[0,1]$, or if  and only if  $(0,\alpha) \cap (\alpha + \{r\} - 1, L + \{r\} - 1) = \varnothing$. This is equivalent to having either $\alpha \leq \alpha + \{r\} -1$, which is impossible, or $L + \{r\} \leq 1.$ That proves part  a).

\medskip
 
%


	Let us prove part  b). By Lemma \ref{L-frac-part} we can assume that $r = \{r\}$, i.e.  that $0 \leq r < 1$. 
	
	By  Theorem \ref{T-Riesz }, $E(\Z)$ is a Riesz sequence on J if and only if the integer translates of J cover $\R$.
	If one of the connected components $[0,\alpha)$ or  $[\alpha + r, L + r)$ covers $\R$ by integer translations, we have that either  $\alpha \geq 1$ or $L - \alpha \geq 1$, and   i) is proved. 
	
	If neither component covers $\R$ by integer translations, i.e. if both $\alpha < 1$ and $L - \alpha < 1$, we can consider   2 sub-cases:
	
\begin{itemize}\item If $r = 0$, we have that  $J=[0, L)$, and  the integer translates of $J$ cover $\R$ if and only if $L \geq 1$; that proves  ii).

	\item Suppose next that $r > 0$.
    The integer translates of $J$ cover $\R$ if and only if the "gap" $(\alpha, \alpha + r)$ is covered by integer translates of $J$. This is possible  if and only if 
    $ (1, \alpha +1) \cup (\alpha + r - 1, L + r -1)\supset (\alpha, \alpha + r) $ (see Figure2).
    
    \begin{figure}[h!]
  \begin{center}
     \begin{tikzpicture}
      %
      
\draw[  black, ->] (-.5 ,0 )--  ( 5 , 0 ) ;

    \draw[  black, thick] (0,0)--  ( 1.5, 0); 
   \draw[  black, thick] (  2.5,0)--  (   3.5, 0 );  
    
    \draw[  black, thick] (1,1)--  ( 2.5, 1); 
   \draw[  black, thick] (  3.5,1)--  (   4.5, 1);  
  
  \draw (  0,-.2 ) node [black ]   {{\tiny $  0 $}};
  
  \draw (  .2, .2 ) node [black ]   {{\tiny $  J $}};
  
\draw (  1.5, -.2) node [black ]   {{\tiny $   \alpha  $}};
\draw (2.5, -.2  )  node [black ]   {{\tiny $ r+\alpha  $}};
\draw (3.5,-.2  )  node [black ]   {{\tiny $ r+L  $}};
\draw (5,1 )  node [black ]   {{\tiny{\bf $J+1 $}}};

 \draw [black,fill]  (1,1 ) circle [radius=0.06];
 \draw [black,fill]  (2.5,1 ) circle [radius=0.06];
 \draw [black,fill]  (3.5,1 ) circle [radius=0.06];
 \draw [black,fill]  (4.5,1 ) circle [radius=0.06];

 \draw[black, dashed] (1,1)--(1,0);
 \draw [black,fill]  (0, 0) circle [radius=0.06];
\draw [black,fill]  (1.5, 0) circle [radius=0.06];
  \draw [black,fill]  (2.5, 0) circle [radius=0.06];
  \draw [black,fill]  (3.5, 0) circle [radius=0.06];
  \draw [black,fill]  (1, 0) circle [radius=0.06];
  \draw (1,-.2  )  node [black ]   {{\tiny $ 1 $}};
    \end{tikzpicture}
 \caption{ } 
   \end{center}
 \end{figure}
    We have $(1, \alpha + 1) \cap (\alpha, \alpha + r) = (1, \alpha + r)$, because $\alpha , r < 1$. Thus, J covers $\R$ if and only if 
    $(\alpha + r -1, L + r - 1) \cap (\alpha, \alpha + r) \supset (\alpha, 1)$. This is equivalent to the conditions 
   \begin{center} 
$   \begin{cases}  r - 1 \leq 0 \\ L + r -1 \geq 1 \\ \alpha + r \geq 1 \end{cases} \Longleftrightarrow  \begin{cases} r  \leq 1 \\ L + r  \geq 2 \\ \alpha + r \geq 1 \end{cases}.$
   \end{center}
 Since $\alpha < 1$ and $L - \alpha < 1$  by assumption, and recalling that    $r < 1$,  we can see at once that  the condition $L + r \geq 2$ implies $\alpha + r \geq 1$. Indeed, if $\alpha + r < 1$, then $$L + r = L - \alpha + \alpha + r < L - \alpha + 1 < 2.$$  Thus, the integer translates of $J$ covers $\R$ if and only if $L + r  \geq 2$, and we have iii). The theorem is proved.
    \end{itemize}
\end{proof}

\section{The rotated square}

Let $ Q_h= Q_h (0)=[-\frac{h}{2}, \frac{h}{2}] \times [-\frac{h}{2}, \frac{h}{2}]$ be the  square in $\R^2$ centered at the origin with sides of length $h$. Let $A_\theta=\begin{bmatrix}
        \cos{(\theta)} & \sin{(\theta)} \\
        -\sin{(\theta)} & \cos{(\theta)}
        \end{bmatrix}$ be the matrix of a   rotation by an angle $\theta$, and let  
$ 
    Q_{h,\theta}   =A_{\theta}Q_h(0)
     $
be the square obtained from the rotation of  $Q_h(0)$.  The following theorem offers a complete solution to Problem 2: 

\begin{theorem}
   a)  $E(\Z^2)$ is a Riesz sequence  on  $L^2(Q_{h,\theta})$ if and only if
   $ h\ge 1-\sin(2\theta) $. 
   
   b)  $E(\Z^2)$ is  a frame  on $L^2(Q_{h,\theta})$ if  and only if $h\leq \frac{1}{\sin\theta+\cos\theta}$.
\end{theorem}
\begin{proof}
   We first prove  Part a). 
 Let $P_1= (\frac 12, \frac 12), \ P_2= (-\frac 12, \frac 12)\ P_3= (-\frac 12, -\frac 12)\ P_4= (\frac 12, -\frac 12)$  be the vertices of $\Q_2$. We first find conditions on $h$ and $\theta$  for which the points $P_1$, ...,  $P_4$ lie on the sides of $Q_{h,\theta}$.
  \begin{figure}[h!]
  \begin{center}
     \begin{tikzpicture}
      %
      
\draw[  black, dashed] (-1.5 ,-1.5 )--  ( 1.5 , -1.5 )-- (1.5 , 1.5 )--(-1.5 , 1.5 )--(-1.5 , -1.5 );

   \fill[gray!10](-2.4,0)--  ( 0, -2.4)--  ( 2.4,0)-- (0,  2.4 )--(-2.4,0);
   \draw[  black] (-2.4,0)--  ( 0, -2.4)--  ( 2.4,0)-- (0,  2.4 )--(-2.4,0);
 \draw[  black, dashed] (-1.5 ,-1.5 )--  ( 1.5 , -1.5 )-- (1.5 , 1.5 )--(-1.5 , 1.5 )--(-1.5 , -1.5 );

%
\draw[  black] (-3 ,0)--  ( 0, -3 )--  ( 3 ,0)-- (0,  3  )--(-3 ,0);
\draw (0 ,0) node [black ]   { $ Q_{h,\theta}$};

\draw (1.8, -1.8 ) node [black ]   {{\small $  P_4  $}};
\draw (1.8, 1.8 )  node [black ]   {{\small $  P_1  $}};
\draw (-1.8, 1.8 )  node [black ]  {{\small $  P_2  $}};
\draw (-1.8, -1.8 ) node [black ]  {{\small $  P_3 $}};

\draw (-3.3,0 ) node [black ]   {{\small $  Q_3  $}};
\draw (3.3,0 )  node [black ]   {{\small $  Q_1  $}};
\draw (0, -3.3 )  node [black ]  {{\small $  Q_4  $}};
\draw (0,3.3) node [black ]  {{\small $  Q_2 $}};
\draw (2.5,2 ) node [black ]  {{  $Q_{l(\theta), \theta}$}};

\fill[gray!60] ( 1.5 , 1.5 )--  ( 1.1 ,  1.5 )-- (1.1 , 1.3 )--( 1.5 , 1.3)--( 1.5 ,  1.5 );
\draw[  black, thick] ( 1.5 , 1.5 )--  ( 1.1 ,  1.5 )-- (1.1 , 1.3 )--( 1.5 , 1.3)--( 1.5 ,  1.5 );
 
\fill[gray!60] ( -1.5 , -1.5 )--  (- 1.1 ,  -1.5 )-- (-1.1 , -1.3 )--(- 1.5 , -1.3)--( -1.5 , -1.5 );
\draw[  black, thick] (- 1.5 ,- 1.5 )--  ( -1.1 , - 1.5 )-- (-1.1 , -1.3 )--( -1.5 , -1.3)--( -1.5 ,  -1.5 );

\fill[gray!60] (  1.5 , -1.5 )--  (  1.1 ,  -1.5 )-- ( 1.1 , -1.3 )--(  1.5 , -1.3)--(  1.5 , -1.5 );
\draw[  black, thick] ( 1.5 ,- 1.5 )--  (  1.1 , - 1.5 )-- ( 1.1 , -1.3 )--(  1.5 , -1.3)--(  1.5 ,  -1.5 );

\fill[gray!60] ( -1.5 ,  1.5 )--  (- 1.1 ,   1.5 )-- (-1.1 ,  1.3 )--(- 1.5 ,  1.3)--( -1.5 ,  1.5 );
\draw[  black, thick] (- 1.5 ,  1.5 )--  ( -1.1 ,   1.5 )-- (-1.1 ,  1.3 )--( -1.5 ,  1.3)--( -1.5 ,   1.5 );

\draw [black,fill]  (1.5, 1.5) circle [radius=0.06];
  \draw [black,fill]  (1.5, -1.5) circle [radius=0.06];
  \draw [black,fill]  (-1.5, 1.5) circle [radius=0.06];
  \draw [black,fill]  (-1.5,- 1.5) circle [radius=0.06];
  \end{tikzpicture}
 \caption{ } 
   \end{center}
 \end{figure}
 
 Let $\ell_{1}:  y-\frac 12=\tan(\theta)(x-\frac 12)$,  
 $ \ell_{2}:  y-\frac 12=-\frac 1{\tan(\theta) } (x+\frac 12)$ and $\ell_{3}:  y+\frac 12=\tan(\theta)(x+\frac 12)
 $ 
 be   the equations of the sides of $Q_{h,\theta}$ that contain  the points $P_1$, $P_2$ and $P_3$, resp.
  It is easy to verify that $\ell_2$ intersects $\ell_1$ and $\ell_3$ at the points  
 $ 
 Q_2= \left( -\frac 12\cos(2\theta), \ \frac 12(1-\sin(2\theta))\right)$ and $
 Q_3= \left( -\frac 12(1-\sin(2\theta)) , \ -\frac 12\cos(2\theta)\right)
 $,  
 and that  the length of the segment that join $Q_2$ and $Q_3$ equals to $  l(\theta) =1-\sin(2\theta)$.
 Thus, when $h\ge   l(\theta)$, the set $E(\Z^2)$ is a Riesz sequence on $L^2(Q_{h, \theta})$.
 
 We show that when $h < l(\theta)$ the integer translates of $Q_{h, \theta}$ do not cover  the plane  anymore. Indeed, if $h < l(\theta)$, the four vertices of $\Q_2$ are outside the square  $Q_{l(\theta), \theta}$ and have positive distance from the boundary of $Q_{h,\theta}$.  We can find a small rectangle $R$   with sides parallel to the sides of $\Q_2$ for which $  R+P_j  \subset \Q_2-Q_{h, \theta}$ for every $j$ (see Figure 3).  The integer translates of $Q_{h,\theta}$ cannot cover   the rectangles $R+P_j$ and so the  condition of Theorem \ref{T-Riesz } is not verified. 
 
 \medskip

  \begin{figure}[h!]
  \begin{center}
     \begin{tikzpicture}

  \draw[  black, dashed] (-2, 1)--(3,1)--(3,6)--(-2, 6)--(-2, 1);
   
\draw (3.5,5.2)  node [black ]  {{\small $  \Q_2  $}};
  \draw[black] (1,7)--(4,3)--(0,0)--(-3,4)--(1,7);
  \fill[gray!60](.5,6)--(.5, 6.3)--(.2, 6.3)--(.2, 6)--(.5, 6);
  \draw[black ] (.5,6)--(.5, 6.3)--(.2, 6.3)--(.2, 6)--(.5, 6);
  \draw (.1,6.5)  node [black ]  {{\small ${  R } $}};
  \fill[gray!30](1,6)--(-2,4)--(0,1)--( 3,3)--(1,6);
  
  \draw[black] (1,6)--(-2,4)--(0,1)--( 3,3)--(1,6);
  
  \draw (1, 6.4) node [black ]   {{\small $  Q_1  $}};
\draw (-2.4, 4)  node [black ]   {{\small $  Q_2 $}};
\draw (0, .6 ) node [black ]  {{\small $  Q_3 $}};
\draw (3.4, 3) node [black ]  {{\small $  Q_4 $}};
  \draw (0,3) node [black ]  {{  $Q_{s(\theta), \theta}$}};
  \draw [black,fill]  (1,6) circle [radius=0.06];
  \draw [black,fill]  (-2,4) circle [radius=0.06];
  \draw [black,fill] (0,1)circle [radius=0.06];
  \draw [black,fill] (3,3)circle [radius=0.06];

  \end{tikzpicture}
 \caption{ } 
   \end{center}
 \end{figure}

 Let us prove   Part b): 
 the vertices $Q_1$, ..,\, $Q_4$ of $Q_{h,\theta}$ lie  on the sides of $\Q_2$ if  and only if there exists $0<t\leq 1$ for which $Q_1=(t-\frac 12,\, \frac 12)$, $Q_2= (-\frac 12,\, t-\frac 12)$, $Q_3= ( \frac 12  -t,\, -\frac 12)$ and $Q_4= ( \frac 12,  \,  \frac 12 -t)$.
 If we let $\tan(\theta)$ be  the slope of the line that joins $Q_3$ and $Q_4$, we can see at once  that
 $  \tan(\theta)= \frac {1-t}{t}$; thus,  $ t=\frac{1}{1+\tan(\theta)}= \frac{\cos(\theta)}{\sin(\theta)+\cos(\theta)}.$ 
 The length of  the segment $[Q_3Q_4]$ is then:
 $$
s(\theta)= \sqrt{ t^2+(1-t )^2}= \frac{1}{\sin(\theta)+\cos(\theta)} $$
and   $E(\Z^2)$ is a frame on $ Q_{h,\theta}$  whenever $h\leq s(\theta)$. 
 
Let us show that $E(\Z^2)$ is not a frame on $Q_{h,\theta}$ whenever    $h>s(\theta)$. Indeed, if  $h>s(\theta)$,  the set  $Q_{h,\theta}-\Q_2$ has positive measure; we can find a small rectangle $R$  with sides parallel to the sides of $\Q_2$  and vectors $n_1$, ..., $n_4\in\Z^2$ such that $R+n_j\subset    Q_{h,\theta}-\Q_2$ (see Figure 3).  Thus, the integers translates of $Q_{h,\theta}$ overlap on the rectangles $R+n_j$ and  by Theorem \ref{T-frame}, $E(\Z^2)$ is not a frame on $L^2(Q_{h,\theta})$.

    \end{proof}

 \section{The translated parallelepiped} In this section we solve  Problem 3.  
 
 Let  $P\subset \R^d$ be a parallelepiped with sides parallel to  vectors $  v_1$, ...,\, $  v_d$.  
 We let $A=\{a_{i,j}\}_{1\leq i,j\leq d}$ be the matrix whose columns are $ v_1$, ...,\, $ v_d$; we let  $A^{-1}=\{b_{i,j}\}_{1\leq i,j\leq d}$.
 We prove the following 
 \begin{theorem}\label{T-frame-parall}
a) The set  $E(\Z^d)$  is a frame  on  $L^2(P)$ if and only if  $\det(A) \leq 1$ and  $\max_{{1\leq i,k, j\leq d}\atop{j\ne k}}|a_{i,j}|+  |a_{i,k}|\leq 1$ 
 
 b) The set $E(\Z^d)$  is a Riesz sequence   on  $L^2(P)$  if and only if $\det(A ) \ge1$ and   $\max_{1\leq i, j\leq d} |b_{i,j}| \leq 1$.
 \end{theorem}

    \begin{figure}[h!]
  \begin{center}
     \begin{tikzpicture}

  \draw[  black, ->] (-1, 0)--(7,0);
  \draw[  black, ->] (0,-1)--(0, 5);
  
  \draw[black] (0,0)--(2,2)--(3.5,2)--(1.5,0)--(0,0);
  \draw[black] (1.8 ,1.5)--(3.8 ,3.5)--(5.3 ,3.5)--(3.3 ,1.5)--(1.8 ,1.5);
 \draw(1, -.3) node[black] {$v_1$};
 \draw(1.8, -.3 ) node[black] {$1$};
 \draw (1.1  ,.5) node [black ]   {{\small $ P $}};
  \draw (6 ,2.8) node [black ]   {{\small $ P +(1,1) $}};
  \draw[black, dashed] (1.8, 1.5)-- (1.8,0);
  \draw[black, dashed] (1.8, 1.5)-- (0,1.5);
  
  \draw(-.2, 1.5 ) node[black] {$1$};
  \draw(1.7, 2 ) node[black] {$v_2$};

  \end{tikzpicture}
 \caption{ } 
   \end{center}
 \end{figure}
 
 \begin{proof}
  Observe that  if $|P|= \det(A) > 1$, the set $E(\Z^d)$ cannot be a frame on $L^2(P)$ so in part a)  we assume $\det(A) \leq 1$. Similarly, for  part b) we assume that $\det(A) \ge 1$.

\medskip
We prove part a)   by induction on the dimension $d$.  

By Theorem \ref{T-frame}, the set $E(\Z^d)$ is a frame on $L^2(P)$  if and only if  the integer translates  of $P$  overlap only on  sets of zero measure. 
In dimension $d=2$,   we let     $ v_1=(a_{1,1}, a_{2,1})$ and $ v_ 2=(a_{1,2}, a_{2,2})$ be the vectors that are parallel to the sides of $P$. When the components of $v_1$ and $v_2$ are non-negative, we can easily verify that  $P$ overlaps with $P+(1,1)$ if and only  if the sum   of the   projections of $v_1$ and $v_2$  on the $x_1$  and $x_2$ axes has measure  $\ge 1$ (see Figure 4).   Thus,    $P$  overlaps with $P+(1,1)$  if and only if   $  a_{1,1} + a_{1,2} \ge 1$ and  $a_{2,1} + a_{2,2} \ge 1$.   These conditions imply that  no pair of integer translates of $P$ intersect. 
For general   $v_1$ and $v_2$ we can similarly verify that the integer translates of $P$ do not intersect  if and only if $ | a_{1,1} |+ |a_{1,2}| \ge 1$ and  $|a_{2,1}| + |a_{2,2} |\ge 1$.

 We now assume that part a) of the theorem is valid  in dimension $d\ge 2$. We prove that is is valid also in dimension $d+1$. 
 
 Let $P$ be a parallelepiped in $\R^{d+1}$.
  The integer translates of $P$ overlap on sets of positive measure in $\R^{d+1}$ if and only if the integer translates of the faces of $P$   overlap   on sets of positive measure in $\R^{d}$.
Let  $P_h$ be the face of $P$ spanned by the vectors  $ v_1$, ..., $ v_{h-1}, \  v_{h+1}$, ...,\, $ v_{d+1}$. 
  
  Let $ e_1=(1, 0, ...,\,0)$, ..., $ e_{d+1} =(0,  ...0,\,1)$ be the standard orthonormal basis in $\R^{d+1}$ and let $H_j$ be the orthogonal complement of  $e_j$. Clearly, the  integer translates of $P_h$ overlap  if and only if  the integer translates of the  orthogonal projections  of $P_h$  on  the  $H_j$'s overlap. 
  
 The projection of  $P_h$ on $H_k$ is a parallelepiped in $\R^d$ spanned by the vectors 
 $ w_1$, ..., $ w_{h-1}, \  w_{h+1}$, ...,\, $ w_{d+1}$ where $ w_j$ is the projection of $v_j$ on $H_k$,  i.e., it is  the vector  $ v_j$ with the $k$-th  component removed.  
  By assumptions,   $\max_{{1 \leq i,k, j\leq d+1\atop{i\ne k}}\atop{k\ne j\ne h}}\{|a_{i,j }| + |a_{i,k }|  \}\leq 1$. 
This  inequality  is valid for every face of $P$ and for every projection, and so we have that $\max_{{1 \leq i,j, k\leq d+1 }\atop{k\ne j}}\{|a_{i,k}|+ |a_{i,j}|\}\leq 1$ as required.

 \medskip
 We now  prove   part b).  By Theorem \ref{T-Riesz }, the   integer translates of $P$ must cover $\R^d$.  
Since  $P$ is   the image of the unit cube $[0,1]^d$ via the linear transformation $A(x)= A  x$, we can write  $P=A([0,1]^d)$.  Thus, $E(\Z^d)$ is a Riesz sequence in $L^2(P)$ if and only if 
 $  \bigcup_{n\in\Z^d} (A([0,1]^d) +  n) =\R^d $, or:
 $$
 A^{-1}\left(\bigcup_{n\in\Z^d} (A([0,1]^d) +  n)\right)= \bigcup_{n\in\Z^d} ( [0,1]^d  + A^{-1} n)=\R^d.$$
 The translates of the unit cube $[0,1]^d$ cover $\R^d$ if and only if the components of the vectors $A^{-1}e_k$  are  all  $\leq 1$, i.e.,  if and only if $\max_{1\leq i,j\leq d}|b_{i,j}|\leq 1$.

 \end{proof}

    \subsection{The shortest vector problem}
    
    Let $A:\R^d\to\R^d$  be linear and invertible;  consider the parallelepiped $P=A(Q)$, where $Q  = [0,1]^d.$  The  sides of $P$ are parallel to the  columns of  the matrix that represents $A$. 
    
    By Corollary \ref{C-complete}, the  set $E(\Z^d)$ is complete in $L^2(A(Q))$ if and only if   the integer translates of $A(Q)$ do not intersect.
	The integer translates of $A(Q)$ intersect if and only if there are $x,\ y \in Q $ such that $Ax = Ay + n$, for some nonzero $n \in \Z^d$. We can also say that the translates of $A(Q_d)$ intersect if and only if there exist $x, y \in Q$ and $n \in \Z^d$ such that $A^{-1}n = x-y$, i.e., if and only if there exists $n \in \Z^d$ such that $A^{-1}n \in D = \{w \mid \| w\|_{\infty} < 1 \}$.

	\medskip
	These considerations show that    Problem 3 is related to the so-called {\it shortest vector problems} (SVP): Given a lattice $\mathcal{L}$ and a norm $||\ ||$ on $\mathbb{R}^d$, find the minimum length $\lambda = \min_{0 \neq v \in \mathcal{L}}\| v\|$ of a nonzero lattice point.  The SVP is known to be NP-hard (see \cite{Aj}). 
	

The conjectured intractability of the SVP and of other optimization problems on lattices  is central in the  construction of secure lattice-based cryptosystems.  
  For more information on this problem see e.g. \cite{AD}, \cite{Kn} and the references cited in these papers.
	

\medskip
	
	
    
    
  \begin{thebibliography}{999}
 %
  

\bibitem{Aj} Ajtai, M.,   {\it Generating hard instances of lattice problems}. Proc. of the 28th annual ACM symposium on Theory of computing (1996)  99–-108.

\bibitem{AD}  Aggarwal, D.; Dubey, C., {\it Improved hardness results for unique shortest vector problem}. Inform. Process. Lett. 116 (2016), no. 10, 631--637. 

 \bibitem{AS}    Aldroubi, A;  Sun,Q.;  Tang, W.,  {\it Connection between $p-$frames and $p-$Riesz bases
in locally finite SIS of $L^p(\R)$. }
  Proceedings of SPIE - The International Society for Optical Engineering,  February 1970
 \bibitem{Aldroubi}  Aldroubi, A.;    Sun, Q.;    Tang, W., {\it p-Frames and shift invariant subspaces of $L^p$}, J. Fourier Anal. Appl. 7 (2001) 1–-21.

\bibitem{AAC} Agora, E.; Antezana, J.; Cabrelli, C.
{\it Multi-tiling sets, Riesz bases, and sampling near the critical density in LCA groups}.  
Adv. Math. 285 (2015), 454–-477. 

\bibitem{selvan} Antony Selvan,  A.;   Radha, R., {\it  Sampling and reconstruction in shift invariant spaces
on  $\R^d$}. Ann. Mat. Pura Appl.  194  (2015) no.  6,  1683--1706.

\bibitem{BHM}  Barbieri, D.; Hernández, E.; Mayeli, A., {\it Lattice sub-tilings and frames in LCA groups}. C. R. Math. Acad. Sci. Paris 355 (2017), no. 2, 193–-199.


   
 \bibitem{B}   Beurling, A.,{\it The Collected Works of Arne Beurling}, Contemporary Mathematics, vol. 2, Birkh\"auser Boston Inc., Boston, MA, 1989
 
 \bibitem{B1}  Beurling, A., {\it Local harmonic analysis with some applications to differential operator}. In:
Some Recent Advances in the Basic Sciences, Vol. 1 (Proc. Annual Sci. Conf., Belfer Grad.
School Sci., Yeshiva Univ., New York, 1962–1964),  109–-125.  
 

 \bibitem{CCS}  Casazza, P.;   Christensen, O.;  Stoeva, D.T., {\it Frame expansions in separable Banach spaces}, J. Math. Anal. Appl. 307 (2005) 710--723.

\bibitem  {Cr}   Christiansen, O.,  
{\it An Introduction to Frames and Riesz Bases }, Birkh\"auser, 2003. 
  \bibitem{CDH} Christiansen, O; Deng, B.; Heil;, C., {\it Density of Gabor Frames}, Appl.  Comput.  Harm. An. 7 (1999) 292–-304. 
 
\bibitem{Christiansen}   Christensen, O;   Stoeva, D. T., {\it p-Frames in separable Banach spaces}, Adv. Comput. Math. 18 (2003) 117–-126.


\bibitem{DK} L. De Carli and A. Kumar, {\it Exponential bases on two dimensional trapezoids},  Proc. Amer. Math. Soc. 143 (2015), no. 7, 2893--2903.


 \bibitem{F}  Fuglede, B., {\it Commuting self-adjoint partial differential operators and a group theoretic
 problem}, J. Funct. Anal. 16 (1974), 101--121.
 
\bibitem{GLi} Gabardo, J.P; Li, Y.Z., {\it Density results for Gabor systems associated with
periodic subsets of the real line}, J. of Approx. Theory 157 (2009) 172–-192.
 

   %
  \bibitem{GL}  Grepstad, S.;  Lev, N., {\it Multi-tiling and Riesz bases}. Adv. Math. 252 (2014), 1--6.

  %
   \bibitem{Haase} Haase, M., {\it Functional Analysis: An Elementary Introduction}, AMS Graduate studies in Mathematics, no. 156, 2014

\bibitem{Heil}  Heil, C., {\it  A basis theory primer}, Appl.  Num. Harm. Analysis,
 Birkh\"auser  (2011).
 
 \bibitem{Kn}  Khot, S., {\it Hardness of approximating the shortest vector problem in lattices}. J. ACM. 52 (2005) no.(5), 789–-808. 
 
  \bibitem{K}  Kolountzakis, M., {\it Multiple lattice tiles and Riesz bases of exponentials,}
Proc. Amer. Math. Soc. 143 (2015), 741--747.

\bibitem{K2} Kolountzakis, M., {\it The study of translational tiling with Fourier analysis}. Fourier analysis and convexity, 131–-187, in: Appl. Numer. Harmon. Anal., Birkhäuser Boston, Boston, MA, 2004.
 
 \bibitem{KRS} Kulkarni, S.H.; Radha, R.; Sivananthan, S.,{\it  Non-uniform sampling problem}. J. Appl. Funct. Anal. 4  (2009), no. 1,
58–-74.
  
 
 \bibitem{JM}   Jia,  R.Q.;  Micchelli,C. A., {\it Using the refinement equation for the construction of pre-wavelets II: power
of two}, In  "Curves and Surfaces" (P. J. Laurent, A. Le Mehaute and L. L. Schumaker eds.), Academic Press,
New York 1991,   209--246.

 
 \bibitem{Laba} Laba, I. {\it Fuglede's conjecture for a union of two intervals}. Proc. Amer. Math. Soc. 129 (2001), no. 10, 2965–-2972.
 
 \bibitem{Landau}  Landau, H. J., {\it Necessary density conditions for sampling and interpolation of certain entire functions}. Acta Math. 117 (1967) 37–-52.
 
 
  \bibitem{NO}  Nitzan, S.;  Olevskii, A., {\it Revisiting Landau’s density theorems for Paley-
Wiener spaces}. C. R. Acad. Sci. Paris 350 (2012) no.  9–10, 509–-512.


\bibitem{S} Seip, K., {\it On the Connection between exponential bases and certain related sequences in $L^2(− 
 \pi, \pi)$}, J.  Funct. Anal.
 130 (1995), no. 1,  131--160.

 
\bibitem{PW}   Paley, R.;  Wiener, N., {\it Fourier Transforms in the Complex Domain},  Amer. Math.
Soc. Colloq. Publ., 19, Amer. Math. Soc., New York, 1934 .
 
    
 
 \bibitem{Y} Young, R. M. {\it An Introduction to Nonharmonic Fourier Series}, Academic Press, New York,
1980.
 
  
\end{thebibliography}
\end{document}